\documentclass[12pt,leqno,oneside]{amsart}

\usepackage{amsmath,amsthm,amssymb,amsfonts,ifpdf,xcomment}

\ifpdf
\usepackage[pdftex,pdfstartview=FitH,pdfborderstyle={/S/B/W 1},%
colorlinks=true, linkcolor=blue, urlcolor=blue, citecolor=blue,%
pagebackref=true]{hyperref}
\else
\usepackage[hypertex]{hyperref}
\fi

\usepackage[abbrev,msc-links,nobysame]{amsrefs}

\newtheorem {thm}   {Theorem}
\newtheorem* {thm*}   {Theorem}
\newtheorem* {prp*}   {Proposition}
\newtheorem {lem}      [thm]    {Lemma}
\newtheorem {que}      [thm]    {Question}
\newtheorem {cor}  [thm] {Corollary}
\newtheorem {prp}[thm]  {Proposition}
\newtheorem {dfn} [thm]    {Definition}

\newtheorem {cnj} [thm]    {Conjecture}
\newcounter{AbcT}

\theoremstyle{definition}

\renewcommand{\a}{\alpha}
\renewcommand{\b}{\beta}
\renewcommand{\d}{\delta}

\newcommand{\e}{\varepsilon}
\newcommand{\f}{\varphi}
\newcommand{\g}{\gamma}

\renewcommand{\l}{\lambda}

\newcommand{\s}{\sigma}
\renewcommand{\t}{\theta}

\newcommand{\R}{{\bf R}}

\newcommand{\Q}{{\bf Q}}

\newcommand{\Z}{{\bf Z}}

\newcommand {\cL} {{\mathcal L}}
\newcommand {\cA} {{\mathcal A}}

\newcommand {\cP} {{\mathcal P}}

\newcommand {\cE} {{\mathcal E}}

\newcommand {\cR} {{\mathcal R}}

\newcommand{\wh}{\widehat}
\newcommand{\wt}{\widetilde}

\DeclareMathOperator{\dist}{dist}

\DeclareMathOperator{\Oo}{O}

\title[Self-similar sets and measures]%
{Self-similar sets and measures on the line}
\author{P\'eter P. Varj\'u}
\thanks{
The author has received funding from the European Research Council (ERC) under
the European Union’s Horizon 2020 research and innovation programme (grant agreement
No. 803711).
The author was supported by the Royal Society}

\begin{document}

\begin{abstract}
We discuss the problem of determining the dimension of self-similar sets and measures
on $\mathbf{R}$.
We focus on the developments of the last four years.
At the end of the paper, we survey recent results about other aspects of self-similar
measures including their Fourier decay and absolute continuity.
\end{abstract}

\maketitle

A (self-similar) iterated function system, IFS for short, is a finite collection
\[
\Phi=\{\f_i:i\in\Lambda\}
\]
of contractive similarities of $\R^d$.
A contractive similarity is a map $x\mapsto \l\cdot U x +t$, where $\l\in(0,1)$,
$U\in\Oo(d)$ is a rotation and $t\in\R$.
We call $\l$ the contraction factor of the similarity.
Given such an IFS, there is a unique self-similar set, that is a compact set $K\subset\R^d$
such that
\[
K=\bigcup_{i\in\Lambda}\f_i(K).
\]
This set $K$ is also known as the attractor of the IFS.
Furthermore, given an IFS and a probability vector $\{p_i:i\in\Lambda\}$, there
is a unique self-similar measure, that is a probability measure $\mu$ on $\R^d$ such that
\[
\mu=\sum_{i\in\Lambda} p_i \f_i(\mu).
\]
Here $\f_i(\mu)$ denotes the push-forward of $\mu$ under $\f_i$.
In other words, $\mu$ is the unique stationary measure for the Markov chain on $\R^d$ with
transitions $\f_i$ executed with probability $p_i$.
The support of $\mu$ equals the self-similar set $K$ provided $p_i>0$ for all $i$.

Self-similar sets and measures are central objects of interest in fractal geometry and
they include many classical examples of fractals.
For example, the attractor of the IFS
\[
\{x\mapsto \l x-1,x\mapsto \l x+1\}
\]
is (a scaled copy of) the middle $1-2\lambda$ Cantor set for $\l\in(0,1/2)$,
while for $\l\ge1/2$, the attractor is an interval.
The self-similar measure associated to the same IFS with equal probability weights $p_i=1/2$
is called the Bernoulli convolution and is denoted by $\nu_\l$.
They can also be defined as the distribution of the random variables
\[
\sum_{n=0}^\infty \pm \l^n,
\]
where the $\pm$ are independent fair coin tosses.
The study of these measures go back at least to Wintner and his collaborators in the 1930's.
See \cite{PSS-60y} for more on Bernoulli convolutions.
Other classical self-similar sets include the Sierpi\'nski triangle and (a side of)
the Koch snowflake curve.

The systematic study of self-similar sets and measures was initiated by Hutchinson
\cite{Hut}.
We refer to his paper and Falconer's book \cite{Fal}*{Chapter 9} for
thorough treatments of the fundamental properties of these objects.

Determining the dimension of self-similar sets and measures is a central problem in fractal
geometry.
While there are several competing notions of dimension for sets and measures, most of them
coincide in the self-similar case.
In this paper, for self-similar sets, by dimension we mean the common value of the Minkowski and
Hausdorff dimensions.

The local dimension of a measure $\mu$ in $\R^d$ at a point $x$ is
\[
\lim_{r\to 0}\frac{\log\mu(B(x,r))}{\log r},
\]
provided the limit exists,
where $B(x,r)$ is the ball of radius $r$ around $x$.
We say that the measure is exact dimensional if its local dimension exists
and is constant $\mu$-almost everywhere.
By the dimension of an exact dimensional measure $\mu$ we mean this $\mu$-almost
constant value of its local dimension.
It is known that self-similar measures are exact dimensional (see \cite{FH-exact-dim}).

Before we state the main conjectures in the dimension theory of self-similar
sets and measures on $\R$, which will be the main focus of this paper,
we make some simple observations to motivate them.
Let $K$ be the attractor of a self-similar IFS $\{\f_i:i\in\Lambda\}$.
We write $H^s$ for the $s$ dimensional Hausdorff measure.
Now suppose that the sets $\f_i(K)$ are pairwise disjoint for $i\in\Lambda$
and that $0<H^s(K)<\infty$ for some $s$.
Then we can write
\[
H^s(K)=\sum_{i\in\Lambda} H^s(\f_i(K))
=H^s(K)\sum_{i\in\Lambda}\lambda_i^s,
\]
where $\l_i$ is the contraction factor of $\f_i$.
It follows that $s$ must be the unique solution of
\begin{equation}\label{eq:sim-dim}
1=\sum_{i\in\Lambda}\l_i^s.
\end{equation}

While an $s$ with $0<H^s(K)<\infty$  may not exist in general, if it does, then
it must equal the Hausdorff dimension of $K$.
Therefore, the above considerations suggest that a reasonable guess for
$\dim(K)$ is the unique solution of the equation \eqref{eq:sim-dim}.
It is a classical result going back to Moran \cite{Mor} in some form,
that this guess is correct when
the IFS satisfies the so-called open set condition, which is a mild relaxation
of requiring that the sets $\f_i(K)$ are pairwise disjoint.
See \cite{Fal}*{Chapter 9} for a precise definition.

It turns out that the unique solution of \eqref{eq:sim-dim} is always an upper bound
for $\dim(K)$ and it is natural to ask to what extent it is possible to
drop the open set condition without turning this upper bound into a strict
inequality.
There are two immediate obstructions to this.
First, the solution of \eqref{eq:sim-dim} may be larger than $d$, but the dimension of
$K$ will never exceed $d$ which is the dimension of the ambient space $\R^d$.
Second, \eqref{eq:sim-dim} depends on the IFS and not only on the set $K$.
It may be possible to realize $K$ as the attractor of another IFS such that
the corresponding \eqref{eq:sim-dim} has a smaller solution.
This happens, for example, if the IFS contains exact overlaps, which we define now.

\begin{dfn}
An IFS $\{\f_i:i\in\Lambda\}$ contains exact overlaps if
there is some $n\in\Z_{\ge 1}$ and
$(i_1,\ldots,i_n)\neq(\wt i_1,\ldots,\wt i_n)\in\Lambda^n$
such that
\begin{equation}\label{eq:overlap}
\f_{i_1}\circ\ldots\circ \f_{i_n}
=\f_{\wt i_1}\circ\ldots\circ \f_{\wt i_n}.
\end{equation}
\end{dfn}

In other words, the IFS contains no exact overlaps if and only if the
semigroup generated by the maps in the IFS with respect to the composition
operation is free.
We note that it does not make a difference in the definition whether or
not we require that we have the same number of composition factors
on the two sides of \eqref{eq:overlap}.

The next Conjecture due to Simon (see \cite{Sim}) predicts that apart
from the above two obstructions, $\dim(K)$ equals the unique solution
of \eqref{eq:sim-dim} in the $d=1$ case.

\begin{cnj}\label{cn:EO-set}
	Let $K$ be the attractor of a self-similar IFS $\{\f_i:i\in\Lambda\}$
	on $\R$
	that contains no exact overlaps.
	Let $\l_i$ be the contraction factor of $\f_i$.
	Then
	\[
	\dim K =\min(1,s),
	\]
	where $s$ is the unique solution of the equation
	\[
	\sum_i\l_i^s=1.
	\]
\end{cnj}

The conjecture also has a counterpart for measures.

\begin{cnj}\label{cn:EO-measure}
	Let $\mu$ be the self-similar measure on $\R$ associated to an $IFS$
	$\{\f_i:i\in\Lambda\}$ without exact overlaps and a
	probability vector $\{p_i\}$.
	Let $\l_i$ be the contraction factor of $\f_i$.
	Then
	\[
	\dim\mu =\min \Bigg(1,
	\frac{\sum_i p_i\log p_i^{-1}}
	{\sum_i p_i\log \l_i^{-1}}\Bigg).
	\]
\end{cnj}
Self-similar measures are of interest in their own right, but a
major motivation for Conjecture \ref{cn:EO-measure} is that it implies
Conjecture \ref{cn:EO-set}.
To see this, recall that if a set $K$ supports an
exact dimensional measure $\mu$ of dimension $s$, then
the Hausdorff dimension of $K$ is at least $s$, see \cite{Fal}*{Principle 4.2}.
This is a common way of giving lower bounds on the Hausdorff dimension.
Now let $s$ be the solution of \eqref{eq:sim-dim}, and consider the probability
weights $p_i=\l_i^s$.
Observe that this choice yields
\[
s=\frac{\sum_i p_i\log p_i^{-1}}
{\sum_i p_i\log |\l_i|^{-1}}
\]
showing that Conjecture \ref{cn:EO-measure} indeed implies
Conjecture \ref{cn:EO-set}.

Almost all of this paper is concerned only with self-similar measures on $\R$.
Some difficulties arise when one tries to formulate versions of
Conjectures \ref{cn:EO-set} and \ref{cn:EO-measure} for self-similar sets and measures
in higher dimensional ambient spaces due to the presence of affine subspaces
of intermediate dimension.
For a discussion of these issues and results in higher dimension, we refer to
\cite{Hoc-Rd}.

The purpose of this paper is to survey results towards
Conjectures \ref{cn:EO-set} and \ref{cn:EO-measure}.
Since this subject has already been exposed by Hochman in his
ICM lecture in 2018 \cite{Hoc-ICM}, we focus on the developments of the
last four years and discuss earlier results only to the extent
necessary to keep our presentation self-contained.

We will outline some ideas from the proofs of these results; however, we will not
give full details, and some of our discussion will be imprecise.
Our aim is to overview the theory and give insight into the role played
by its components.
For details and a rigorous discussion of the proofs we refer to the original
papers.

In the final section, we briefly survey some further recent developments
on Fourier decay and absolute continuity of self-similar measures.

%%%%%%%%%%%%%%%%%%%%%%%%%%%%%%%%%%%%%%%%%%%%%%%%%%%%%%%%%%%%%%%%%%%%%%%%%%%%%%%
\section{Exponential separation property}\label{sc:exp-sep}
%%%%%%%%%%%%%%%%%%%%%%%%%%%%%%%%%%%%%%%%%%%%%%%%%%%%%%%%%%%%%%%%%%%%%%%%%%%%%%%

The exponential separation property was introduced by Hochman \cite{Hoc-R}
who showed that Conjectures \ref{cn:EO-set} and \ref{cn:EO-measure}
hold when the IFS satisfies this property.
This property can be verified in many cases of interest.
While these results have been already discussed in \cite{Hoc-ICM}, we
recall them now, because they are of crucial importance to later
developments both logically and for the motivation of ideas.

We begin with the definitions.
We introduce a distance function on the group of similarities
on $\R$. Let $\f_i:x\mapsto \lambda_i x+t_i$
be similarities for $i=1,2$.
We define
\[
\dist(\f_1,\f_2)=
\begin{cases}
|t_1-t_2| &\text{if $\l_1=\l_2$},\\
\infty &\text{if $\l_1\neq\l_2$}.
\end{cases}
\]

Given an IFS $\Phi:=\{\f_i:i\in\Lambda\}$, we define its
level $n$ separation by
\[
\Delta_n(\Phi):=\min_{(i_1,\ldots,i_n)\neq(\wt i_1,\ldots,\wt i_n)\in\Lambda^n}
\dist(\f_{i_1}\circ\ldots\circ \f_{i_n}, \f_{\wt i_1}\circ\ldots\circ \f_{\wt i_n}).
\]
We say that the IFS satisfies the exponential separation property
if there is a constant $c>0$ such that $\Delta_n(\Phi)>c^n$ for infinitely
many $n$'s.

We observe that the IFS contains exact overlaps if and only if
$\Delta_n=0$ for some and hence for all sufficiently large $n$.
The exponential separation property is a quantitative strengthening of
the condition that the IFS contains no exact overlaps.
Hochman proved that Conjectures \ref{cn:EO-set} and \ref{cn:EO-measure}
hold under this strengthening of the hypothesis.

\begin{thm}[Hochman \cite{Hoc-R}]\label{th:exp-separation}
Let $\{\f_i:i\in\Lambda\}$ be an IFS that satisfies the exponential
separation property and let $K$ be its attractor.
Write $\l_i$ for the contraction factor of $\f_i$.
Then
\[
\dim K =\min(1,s),
\]
where $s$ is the unique solution of the equation
\[
\sum_i\l_i^s=1.
\]

Let $\mu$ be the self-similar measure associated to the above $IFS$ and a
probability vector $\{p_i\}$.
Then
\[
\dim\mu =\min \Bigg(1,
\frac{\sum_i p_i\log p_i^{-1}}
{\sum_i p_i\log \l_i^{-1}}\Bigg).
\] 
\end{thm}

It can be shown that the exponential separation property holds in parametric families
of IFS's for all but possibly a (packing or Hausdorff) codimension $1$ subset of exceptions.
This shows that Conjectures \ref{cn:EO-set} and \ref{cn:EO-measure} hold generically
in a very strong sense.
We refer to \cite{Hoc-R} for details and more precise results.

We also note that a stronger version of Conjecture \ref{cn:EO-measure}
involving the $L^q$ dimension instead of local dimension of measures
was established subsequently by Shmerkin \cite{Shm}
under the exponential separation property.
This result has very important and far reaching applications, see also \cite{Shm-survey}
and Shmerkin's paper in this volume.

Our main focus here are explicit cases and families of IFS's for which the
exponential separation property can be verified.
We first observe that the exponential separation property holds
always whenever all contraction and translation parameters in the IFS
are rational and the IFS contains no exact overlaps.
Indeed, writing $Q$ for the least common denominator of all
parameters, a simple calculation shows that the translation parameters
of $n$-fold compositions of maps in the IFS have denominators that
divide $Q^n$.
This means that for all $n$, we have $\Delta_n\ge Q^{-n}$ or $\Delta_n=0$.
The second possibility is excluded by the absence of exact overlaps.

In fact, the above reasoning can be extended to the case when the
parameters are algebraic numbers and not necessarily rational.
To do this, one need to work with heights instead of denominators, or
see \cite{Hoc-R}*{Theorem 1.5} for a more elementary argument.
This leads to the following result

\begin{cor}[Hochman]\label{cr:algebraic}
Conjectures \ref{cn:EO-set} and \ref{cn:EO-measure}
hold for IFS's in which all contraction and translation parameters
are algebraic numbers.
\end{cor} 

The exponential separation property can be verified also for certain
IFS's that involve transcendental parameters.
One such example is the family IFS's
\begin{equation}\label{eq:IFS-Sierpinski}
\Big\{x\mapsto \frac{x}{3}, x\mapsto \frac{x}{3}+1, x\mapsto \frac{x}{3}+t\Big\},
\end{equation}
where $t\in\R$ is a parameter.
It can be seen that the attractors of these IFS's are the
linear projections of the Sierpi\'nski triangle.

Another corollary of Theorem \ref{th:exp-separation} is the following.

\begin{cor}[Hochman]\label{cr:Sierpinski}
Conjectures	\ref{cn:EO-set} and \ref{cn:EO-measure} hold for the IFS
\eqref{eq:IFS-Sierpinski} for all values of the parameter $t\in\R$.
\end{cor}

We sketch the proof of the exponential separation property for the family
\eqref{eq:IFS-Sierpinski},
as these ideas will recur later.
For details, see \cite{Hoc-R}*{Theorem 1.6}, where this argument is attributed
to Solomyak and Shmerkin.
The translation component of an $n$-fold composition of maps from the
above IFS is of the form
\[
\sum_{j=0}^{n-1} \a_j 3^{-j},
\]
where each $\a_j$ is equal to $0$, $1$ or $t$.
Based on this observation, it can be seen that for each $t$ and
for each $n$, there
are some integers $a_1,a_2\in\Z$ not both $0$
with $|a_1|,|a_2|\le 3^{n-1}$ such that
\[
\Delta_n=\frac{a_1}{3^{n-1}}-\frac{a_2}{3^{n-1}}t.
\]
Assuming $a_2\neq 0$, which holds whenever $\Delta_n\le 3^{-n+1}$, we get
\[
\Big|t-\frac{a_1}{a_2}\Big|\le 3^n \Delta_n.
\]

Now fix the value of the parameter $t$ such that the IFS \eqref{eq:IFS-Sierpinski} 
contains no exact overlaps.
Suppose $\Delta_n<27^{-n-1}$ for some $n$.
Then there is a rational number $a_1/a_2$ as above such that
$|t-a_1/a_2|<9^{-n-1}$.
Let $\wt n$ be such that $9^{-\wt n-1}< |t-a_1/a_2|\le 9^{-\wt n}$.
(Note that $t\neq a_1/a_2$, for otherwise we would have $\Delta_n=0$ and the
IFS would contain exact overlaps.)
We observe that there is no rational $\wt a_1/\wt a_2$ with
$|t-\wt a_1/\wt a_2|< 9^{-\wt n-1}$ and $|\wt a_1|,|\wt a_2|\le 3^{\wt n-1}$.
Indeed, if such a rational existed, we had
\[
\Big|\frac{a_1\wt a_2-a_2\wt a_1}{a_2\wt a_2}\Big|
=\Big|\frac{a_1}{a_2}-\frac{\wt a_1}{\wt a_2}\Big|
\le \Big|\frac{a_1}{a_2}-t\Big|+\Big|t-\frac{\wt a_1}{\wt a_2}\Big|
\le 2\cdot 9^{-\wt n}.
\]
Since $|a_2\wt a_2|\le 9^{\wt n-1}$, this would yield
$a_1\wt a_2-a_2\wt a_1=0$ leading to $a_1/a_2=\wt a_1/\wt a_2$
contradicting
\[
\Big|t-\frac{\wt a_1}{\wt a_2}\Big|< 9^{-\wt n-1}
<\Big|t-\frac{a_1}{a_2}\Big|.
\]
This shows that $\Delta_{\wt n}\ge 27^{-\wt n-1}$, and the exponential
separation property follows.

A key property of the IFS \eqref{eq:IFS-Sierpinski} exploited in
the above argument is that
exact overlaps occur for certain special values of the parameter
$t$, in this case certain rational numbers, and these special
values are very well separated from each other.
This will be a recurrent concern for us in what follows.

A similar argument can be made when the contraction factor $1/3$ in
\eqref{eq:IFS-Sierpinski} is replaced by another algebraic number.
We omit the details.

%%%%%%%%%%%%%%%%%%%%%%%%%%%%%%%%%%%%%%%%%%%%%%%%%%%%%%%%%%%%%%%%%%%%%%
\section{Bernoulli convolutions}\label{sc:Bernoulli}
%%%%%%%%%%%%%%%%%%%%%%%%%%%%%%%%%%%%%%%%%%%%%%%%%%%%%%%%%%%%%%%%%%%%%%

In this section, we consider the one parameter family of IFS's
\[
\Phi_\l:=\{x\mapsto\lambda x,x\mapsto \lambda x+1\},
\]
where $\l\in(0,1)$.
Instead of $0$ and $1$ we could take any other pair of distinct real
numbers as the translation parameters; we would get the same IFS up to
a change of coordinates.
In fact, it is more customary to take $\pm1$ instead of $0$ and $1$, but
the above choice will make notation more consistent with the rest of this note.

In this case, the resulting self-similar sets have a simple structure.
For $\l< 1/2$, it is the middle $1-2\l$'th Cantor set, while for $\l\ge 1/2$
it is an interval.
In both cases, Conjecture \ref{cn:EO-set} is easily verified.
However, the associated self-similar measures called Bernoulli
convolutions are more difficult to understand.
The purpose of this section is to summarize the developments that lead to
the following result.

\begin{thm}\label{th:Bernoulli}
Conjecture \ref{cn:EO-measure} holds for the IFS
$\Phi_\l$
for any value of the parameter $\l\in(0,1)$.
\end{thm}

For algebraic parameters, this result is due to Hochman as it falls under
the scope of Corollary \ref{cr:algebraic}.
For transcendental parameters, the result has been established in \cite{Var-Bernoulli}.
Strictly speaking, only the case of uniform $(1/2,1/2)$
probability weights is treated there,
but the arguments can be extended to the general case.
Moreover, one can even allow more general
IFS's with an arbitrary number of maps as long as the contraction factors are the
same and the translation parameters are rational.
This has been demonstrated in the Appendix of \cite{RV-3maps}.

To simplify the exposition, we assume in our discussion that the probability weights
are uniform. 
We write $\nu_\l$ for the self-similar measure associated to the IFS $\Phi_\l$. 
We note that $\nu_\l$ is the law of the random variable
$\sum_{n=0}^{\infty}\xi_n\l^n$, where
$(\xi_n)$ is a sequence of independent random variables taking the values $0$
and $1$ with equal probability.

In the algebraic case, Hochman's results yield more information, which allow
to compute the dimension even in the presence of exact overlaps.
This is in terms of the entropy rate of the IFS $\Phi_\l$, which we define now, and
which will also play an important role later on.
The entropy rate is defined as
\[
h(\Phi_\l):=\lim_{n\to\infty}\frac{H(\sum_{j=0}^{n-1}\xi_j \l^j)}{n},
\]
where $H(\cdot)$ stands for Shannon entropy of a discrete random variable.
The numerator on the right can be shown to be a subadditive sequence, hence the
limit exists, and moreover,
\[
h(\Phi_\l)\le\frac{H(\sum_{j=0}^{n-1}\xi_j \l^j)}{n}
\]
for each $n$.

See \cite{BV-entropy}*{Section 3.4}
for the details of how the following follows from the main result of
Hochman \cite{Hoc-R}.

\begin{thm}[Hochman]\label{th:Hoc-algebraic}
Let $\l\in(0,1)$ be an algebraic number.
Then
\begin{equation}\label{eq:alg-dim}
\dim\nu_\l=\min\Big(1,\frac{h(\Phi_\l)}{\log\l^{-1}}\Big).
\end{equation}
\end{thm}

This result together with Theorem \ref{th:Bernoulli} gives an almost complete
solution to the problem of determining the dimension of Bernoulli convolutions.
In addition, there are numerical algorithms to compute $\dim\nu_\l$ with arbitrary
precision for any given algebraic $\l$, see \cite{AFKP}, \cite{HKPS}, \cite{KPV}
and \cite{FF}.
However, it is still not known precisely what is the set
of algebraic parameters $\l\in(1/2,1)$
for which $\dim\nu_\l<1$.

We turn to the case of transcendental parameters in Theorem \ref{th:Bernoulli}.
If the IFS $\Phi_\l$ satisfied the exponential separation property whenever it does
not contain exact overlaps, then Theorem \ref{th:Bernoulli} would follow at once from
Theorem \ref{th:exp-separation}.
This very well could be true; however, this is still an open problem, which seems to
be beyond reach of existing methods.

In fact, the decay rate of $\Delta_n(\Phi_\l)$ is very closely related to a problem
in Diophantine approximation, which is the separation between the elements of the
set
\[
\cE^{(n)}:=
\{\eta:\text{ $P(\eta)=0$ for some polynomial $P\in\cP^{(n)}$}\},
\]
where $\cP^{(n)}$ is the set of polynomials of degree at most $n-1$ with coefficients $-1,0,1$.
As it will be clear from what follows, the set
\[
\cE:=\bigcup_n \cE^{(n)}\cap(0,1)
\]
is precisely the set of parameters for which $\Phi_\l$ contains exact overlaps.

We begin our discussion of the proof of Theorem \ref{th:Bernoulli} by explaining the connection
between the behaviour of $\Delta_n(\Phi_\l)$ and the separation properties of the sets
$\cE^{(n)}$ following Hochman \cite{Hoc-R}*{Question 1.10}.
This can be formalized as follows.

\begin{lem}\label{lm:separation}
If it is true that the elements of $\cE^{(n)}$ are separated by at least $C^{-n}$ for some 
constant $C$ for all $n$, then the exponential separation property holds
for the IFS $\Phi_\l$ whenever it lacks exact overlaps.
\end{lem}

\begin{proof}[Sketch proof]
Fix some $\e>0$ and assume $\l\in(\e,1-\e)$.
We first observe that if $\Delta_n(\Phi_\l)<C^{-n}$ for some $C=C(\e)$, then
there is some $\eta\in\cE^{(n)}$ with
\[
|\l-\eta|<\Delta_n(\Phi_\l)^{\a}
\]
for some $\a=\a(\e)>0$.
This follows from the fact that the translation component of an $n$-fold composition of the
maps in $\Phi_\l$ in some order is a polynomial in $\l$ of degree at most $n-1$ with coefficients
$0,1$.
This means that $\Delta_n(\Phi_\l)=P(\l)$ for some $P\in\cP^{(n)}$ that also depends on $\l$.
To complete the proof of our observation, we need to argue that the only way
$P(\l)$ can be very small is if $\l$ is close to a root of $P$.
For more details see \cite{Var-ECM}*{Lemma 5.2}.

Now suppose that $\l$ is such that $\Delta_n(\Phi_\l)<C_2^{-2n/\a}$ for some $n$,
where $\a$ is as in the previous paragraph and $C_2$ is the constant $C$ in the
assumption about the separation between the elements of $\cE_n$.
Then there is $\eta_n$
such that $|\l-\eta_n|<C_2^{-2n}$.
If $\Phi_\l$ contains no exact overlaps, then $\l\notin \cE$ so $\l\neq \eta_n$.
Now we take the smallest integer $\wt n>n$ such that $|\l-\eta_n|> C_2^{-2\wt n}$.
It follows by the assumed separation property on $\cE^{(\wt n)}$, that there is no
$\eta_{\wt n}\in\cE^{(\wt n)}$ with $|\l-\eta_{\wt n}|< C_2^{-2\wt n}$.
This means that $\Delta_{\wt n}(\Phi_\l)\ge C_2^{-2\wt n/\a}$, and the exponential separation
property follows.
\end{proof}

It is not known whether or not the elements of $\cE^{(n)}$ are exponentially separated.
The best lower bound known for the minimal distance of the elements of $\cE^{(n)}$ is
$\exp(-C n\log n)$ for some constant $C$ (one could take e.g. $C=4$),
which is due to Mahler \cite{Mah}.
This yields via the argument in the proof of Lemma
\ref{lm:separation} that for all $\l$ such that $\Phi_\l$ contains
no exact overlaps, there are infinitely many values of $n$ with
\begin{equation}\label{eq:nadn}
\Delta_n(\Phi_\l)\ge \exp(-Cn\log n)
\end{equation}
for some (other) constant $C$.

One may wonder if this weaker separation condition could be used in a refined
form of Hochman's argument in place of exponential separation.
This has been done in \cite{BV-transcendent}, however, the argument requires that there are
several values of $n$ sufficiently close to each other such that the separation
\eqref{eq:nadn} holds.
Such a condition can be satisfied if we assume that $\l$ is not approximated
too closely by elements of $\cE^{(n)}$.
Indeed, in the above argument the size of $\wt n$ is controlled by the distance between
$\l$ and $\cE^{(n)}$.
More precisely, the following was proved in \cite{BV-transcendent}.

\begin{thm}[Beruillard, Varj\'u]\label{th:BV2}
Let $\l\in(1/2,1)$ be such that Conjecture \ref{cn:EO-measure} does not
hold for $\Phi_\l$.
Then there is $\d>0$ and there are infinitely many values of $n$
such that there is $\eta_n\in\cE^{(n)}\cap(1/2,1)$ with
\begin{align*}
|\l-\eta_n|<\exp(-n^{100}),\\
\dim \nu_{\eta_n}<1-\d.
\end{align*}
\end{thm}

The exponent $100$ can be replaced by any other number, or even by a slowly growing function
of $n$, see \cite{BV-transcendent} for details.
This result along with Theorem \ref{th:exp-separation} are major ingredients in the proof of Theorem \ref{th:Bernoulli}.
Given some $\l\in(1/2,1)$ such that $\Phi_\l$ lacks exact overlaps, it can be shown that
$\l$ has only finitely many approximants $\eta_n$ as in the conclusion of Theorem \ref{th:BV2}
or else $\Phi_\l$ satisfies the exponential separation property.
In either case, Conjecture \ref{cn:EO-measure} follows for $\Phi_\l$ from one of
Theorems \ref{th:exp-separation} or \ref{th:BV2}.

Before we discuss the details of how this can be done, a further remark about Theorem
\ref{th:BV2} is in order.
We have seen that if $\Delta_n(\Phi_\l)<C^{-n}$ for some $n$ and $\l$ with an appropriate
constant $C$, then $\l$ is approximated by some $\eta\in\cE^{(n)}$.
However, we claim some additional properties of this $\eta$
in Theorem \ref{th:BV2}, most importantly
that $\dim\nu_{\eta}<1-\d$.
Now we indicate how this can be deduced.
This leads us to a somewhat lengthy digression; however, it also gives us the opportunity to
introduce several concepts and ideas that will be needed later on.

Already in Theorem \ref{th:exp-separation},
the exponential separation property can be relaxed
(see \cite{Hoc-R}*{Theorems 1.3 and 1.4}).
Instead of assuming
$\Delta_n(\Phi)>C^{-n}$, it is enough to know
that there are not too many pairs of $n$-fold compositions
of maps in $\Phi$ whose translation components are closer than $C^{-n}$.
Likewise in the proof of Theorem \ref{th:BV2}, we work with
a similarly relaxed version of \eqref{eq:nadn}.

To properly quantify this, we use entropy.
Let $X$ be a bounded real valued random variable and let $r\in\R_{>0}$.
The entropy of $X$ at scale $r$ is defined as
\[
H(X;r)=H(\lfloor r^{-1} X\rfloor),
\]
where $H(\cdot)$ on the right is Shannon entropy.
This is the entropy of $X$ with respect to a partition of $\R$ into consecutive
intervals of length $r$.
The choice of this partition is not canonical, and we obtain different values of
$H(X;r)$ by translating $X$.
There are advantages of averaging over translations of $X$ in the definition of
$H(X;r)$, as it is done e.g. in \cite{BV-transcendent}, \cite{Var-ac}  and subsequent papers;
however, we ignore this point here for
the sake of simplicity.

By definition, $\Delta_n(\Phi_\l)>r$ implies
that the points in the support of $\sum_{j=0}^{n-1}\xi_j \l^j$ are separated
by a distance of at least $r$, hence
\[
H\Big(\sum_{j=0}^{n-1}\xi_j \l^j;r\Big) = \log (2) \cdot n.
\]
In the proof of Theorem
\ref{th:BV2}, instead of working with lower bounds on $\Delta_n(\Phi_\l)$
like \eqref{eq:nadn},
we work with bounds of the type
\begin{equation}\label{eq:ent-bound}
H\Big(\sum_{j=0}^{n-1}\xi_j \l^j;r\Big) \ge \b  n
\end{equation}
with suitable $\b$ and $r$.

Now consider some $\l>1/2$ that lacks the approximations $\eta_n$ as
described in the conclusion in Theorem \ref{th:BV2}.
We discuss how this assumption can be used to show that bounds of the type
\eqref{eq:ent-bound} hold for suitably many different values of $n$.
Using such bounds and arguments based on Hochman's proof of Theorem \ref{th:exp-separation},
which we do not discuss in this paper,
it can be shown that $\dim\nu_\l=1$ proving (the contrapositive of)
Theorem \ref{th:BV2}.

In short, the failure of \eqref{eq:ent-bound} with a suitably small $r$
implies that $\l$ can be approximated by some $\eta_n\in\cE^{(n)}$
such that $\Phi_{\eta_n}$ has enough exact overlaps to force
$\dim\nu_{\eta_n}\le \b/\log\eta_n^{-1}$.

We give some more details.
For every pair of numbers $x_1,x_2$ in the support of $\sum_{j=0}^{n-1}\xi_j \l^j$
such that $|x_1-x_2|\le r$, there is a polynomial $P\in\cP^{(n)}$ such that
\[
|x_1-x_2|=|P(\l)|\le r.
\]
As we have already seen, all such polynomials have a root near $\l$ provided
$r<C^{-n}$ for a suitable constant $C$.
If $r<\exp(-Cn\log n)$ for another suitable $C$, then all the roots obtained
this way as $(x_1,x_2)$ goes over all pairs of points in the
support of $\sum_{j=0}^{n-1}\xi_j \l^j$
that are at distance not more than $r$ can be shown to coincide.
This follows from Mahler's aforementioned bound on the separation of elements
in $\cE^{(n)}$.
For an alternative argument, see \cite[Section 3]{BV-transcendent}.
 
Now it follows that if
\[
H\Big(\sum_{j=0}^{n-1}\xi_j \l^j;r\Big) < \log(2)\cdot  n
\]
for some $r<\exp(-Cn\log n)$, then there is some $\eta_n\in\cE^{(n)}$ close to $\l$
(the common root of the polynomials discussed in the previous paragraph) such that
\begin{equation}\label{eq:ent-eta}
H\Big(\sum_{j=0}^{n-1}\xi_j \eta_n^j\Big)\le H\Big(\sum_{j=0}^{n-1}\xi_j \l^j;r\Big).
\end{equation}
Notice that on the left, there is no designated scale, so $H(\cdot)$ stands
for Shannon entropy there.
Provided $H(\sum_{j=0}^{n-1}\xi_j \l^j;r)$ is sufficiently small, this can be
turned into a bound on $\dim \nu_{\eta_n}$ with the help of Theorem \ref{th:Hoc-algebraic}.
Indeed, combining our observations, we see that
\[
H\Big(\sum_{j=0}^{n-1}\xi_j \l^j;r\Big) \le \b  n
\]
implies
\[
\dim\nu_{\eta_n}\le \frac{h(\Phi_{\eta_n})}{\log\eta_n^{-1}}
\le \frac{H(\sum_{j=0}^{n-1}\xi_j \eta_n^j)}{n\log\eta_n^{-1}}
\le \frac{H(\sum_{j=0}^{n-1}\xi_j \l^j;r)}{n\log\eta_n^{-1}}
\le\frac{\b}{\log\eta_n^{-1}}.
\]

By the assumption that $\l$ lacks the approximations as in the conclusion
of Theorem \ref{th:BV2}, we conclude $|\l-\eta_n|>\exp(-n^{100})$.
As we have already discussed, this implies that we can find an $\wt n$
not larger than $n^{100}$ such that even \eqref{eq:nadn} holds
with $\wt n$ in place of $n$.
This provides a sufficiently plentiful supply of numbers $n$ such that
at least a bound of the type \eqref{eq:ent-bound} holds.

We return to the proof of Theorem \ref{th:Bernoulli}.
We suppose to the contrary that $\l\in(1/2,1)$
is a counterexample to Conjecure \ref{cn:EO-measure}.
By Theorem \ref{th:BV2}, there are infinitely many approximants $\eta_n$ to $\l$
satisfying the conclusion of that theorem.
We fix such an $\eta_n$ corresponding to a suitably large $n$.

By virtue of \eqref{eq:alg-dim}, we have 
$h(\Phi_{\eta_n})\le(1-\d)\log\eta_n$.
Our next step is to convert this information to something that is easier to exploit
with the methods of Diophantine Approximation.
We introduce a definition for this purpose.
The Mahler measure of an algebraic number $\eta$ with minimal polynomial
$a_d(x-\eta^{(1)})\cdots(x-\eta^{(d)})\in\Z[x]$ is defined as
\[
M(\eta)=|a_d|\prod_{j=1}^d\max(1,|\eta^{(j)}|),
\]
i.e. it is the product of the absolute values of the leading coefficient and
the roots outside the unit disk.
This quantity is widely used in number theory as a measure of the `complexity' of
$\eta$.
Notice that if $\eta\in\Q$, then $M(\eta)$ is the maximum of the absolute values of
the numerator and the denominator of $\eta$.

Breuillard and Varj\'u \cite{BV-entropy} found a connection between the entropy rate and
the Mahler measure.
A form of this most suited for the proof of Theorem \ref{th:Bernoulli} is the following.

\begin{thm}[Breuillard, Varj\'u]\label{th:BV1}
For any $h\in(0,\log 2)$, there is a number $C(h)$ such that $h(\Phi_\eta)\le h$
implies $M(\eta)<C(h)$ for all algebraic numbers $\eta$.
\end{thm}

See \cite{Var-Bernoulli}*{Theorem 9} for the details of how this
follows from the technical results
of \cite{BV-entropy}.

Using this theorem, we conclude that $M(\eta_n)<C$ for a constant $C$ that only depends on
$\l$, but not on $n$.
Furthermore, recall that we have $|\l-\eta_n|<\exp(-n^{100})$.
Now we use the following, which follows easily from a more general result of
Mignotte \cite{Mig}.

\begin{thm}[Mignotte]\label{th:Mignotte}
	Let $\eta$ be an algebraic number of degree at most $n$.
	Let $\wt n>n(\log n)^2$ be an integer, and let $\wt \eta\neq \eta\in\cE^{(\wt n)}$.
	Then there is an absolute constant $C$, such that
	\[
	|\eta-\wt\eta|\ge C^{-\wt n} M(\eta)^{-2\wt n}.
	\]
\end{thm}

We finish our discussion of the proof of Theorem \ref{th:Bernoulli}.
Thanks to the approximation of $\l$ by $\eta_n$ this theorem acts as a substitute
for the separation condition between elements of $\cE^{(\wt n)}$ in the proof
of Lemma \ref{lm:separation},
and
we can conclude that $\Delta_{\wt n}(\Phi_\l)>C^{-\wt n}$ for a suitable choice of $\wt n$
for some $C$ independent of $n$.
Now we are in a position to apply Theorem \ref{th:exp-separation} to show that
Conjecture \ref{cn:EO-measure} holds for $\l$, which is our desired contradiction
proving Theorem \ref{th:Bernoulli}.

The original argument in \cite{Var-Bernoulli} used an alternative variant of Theorem \ref{th:Mignotte}, which was
deduced from an observation of Garsia \cite{Gar-arithmetic}
and a transversality argument of Solomyak \cite{Sol}.
It was pointed out by Vesselin Dimitrov that the transversality argument can be replaced
by a simpler version based on Jensen's formula.
This has the advantage that it is applicable in greater generality.
See \cite{RV-3maps}*{Lemmata 2.3 and 4.6} for details.

%%%%%%%%%%%%%%%%%%%%%%%%%%%%%%%%%%%%%%%%%%%%%%%%%%
\section{Failure of exponential separation}\label{sc:failure}
%%%%%%%%%%%%%%%%%%%%%%%%%%%%%%%%%%%%%%%%%%%%%%%%%%

As we discussed in the previous section, it is not known whether Bernoulli convolutions
without exact overlaps satisfy the exponential separation property.
However, they are known to satisfy a slightly weaker lower bound on $\Delta_n$,
and this played an important role in the proof of Conjecture \ref{cn:EO-measure}
for this class of IFS's.

On the other hand, there are some IFS's without exact overlaps
for which it is known that the exponential
separation property fails, and moreover, $\Delta_n$ converges to $0$ in an arbitrarily
fast prescribed way.

\begin{thm}[Baker; B\'ar\'any, K\"aenm\"aki]\label{th:BBK}
Let $(\eta_n)\subset \R_{>0}$.
Then there is an IFS $\Phi$ without exact overlaps
such that $\Delta_n(\Phi)\le \eta_n$ for all $n$.
\end{thm}

The first examples of such IFS's were given by Baker \cite{Bak-fail1} in the form
\[
\Big\{
x\mapsto \frac{x}{2},
x\mapsto \frac{x+1}{2},
x\mapsto \frac{x+s}{2},
x\mapsto \frac{x+t}{2},
x\mapsto \frac{x+1+s}{2},
x\mapsto \frac{x+1+t}{2}
\Big\}
\]
for suitable choices of the parameters $t,s$,
and by B\'ar\'any, K\"aenm\"aki \cite{BK-fail} in the form
\[
\{
x\mapsto\l x,
x\mapsto\l x+1,
x\mapsto \l x+t
\}
\]
for suitable choices of $\l,t$.
Baker's example was modified by Chen \cite{Che-fail}, who disposed of the last two maps
and replaced the denominator $2$ by an arbitrary real algebraic number not smaller than $2$.
These constructions were further extended by Baker \cite{Bak-fail2}.

In what follows we give a heuristic argument to show why such IFS's with very small
separation may be expected to exist.
Our purpose (due to limitation of space) is not to give insight to the proofs of Theorem \ref{th:BBK}, which 
are based on a variety of tools, such as continued fraction expansions
in \cite{Bak-fail1} and the transversality method in \cite{BK-fail}.
Instead, we just aim to highlight the difference between families of IFS's depending on a single
parameter, such as Bernoulli convolutions, or the examples covered by Corollary \ref{cr:Sierpinski},
and families depending on more than one parameter, which will be discussed in the
next two sections.

Let
\[
\Phi_{x,y}=\{\f_{i,x,y}:i\in\Lambda\}
\] be a family of IFS's (smoothly) depending on two parameters.
Let $n\in\Z_{>0}$, and we write $\Gamma^{(n)}$ for the collection
of curves in the parameter space, which arise as the solution sets of equations of the form
\[
\f_{i_1,x,y}\circ\ldots\circ \f_{i_n,x,y}
=\f_{\wt i_1,x,y}\circ\ldots\circ \f_{\wt i_n,x,y}
\]
in $(x,y)$
where $i_1,\ldots,i_{n}$ and
$\wt i_1,\ldots,\wt i_{n}$ are two distinct sequences of indices in $\Lambda$.
Note that the union of all these curves is the set of all parameter points for which
the IFS contains exact overlaps.

The key difference between this setting and a
family depending on a single parameter
is that exact overlaps occur along curves in the parameter space rather than at isolated
points.
These curves may intersect each other, and then there is no separation between them,
which rules out the arguments presented for the proof of Corollary \ref{cr:Sierpinski} and later
in Section \ref{sc:Bernoulli}.

We now give the heuristic suggesting the existence of the IFS's claimed in Theorem \ref{th:BBK}.
We give a recursive construction.
After the $k$'th step, we will have a sequence
$n_1,\ldots,n_k\in\Z_{\ge 1}$, a sequence $\g_1,\ldots,\g_k$, where $\g_j$ is a
segment of a curve in $\Gamma^{(n_j)}$, and a sequence $\d_1,\ldots,\d_{k-1}\in\R_{>0}$.
These will satisfy the property that $\g_k$ is contained in the $\d_j$ neighbourhood of
$\g_j$ for all $j<k$.

We begin the process by setting $\g_1$ to be any segment (of positive length)
of a curve in $\Gamma^{(1)}$.
Suppose now that $\g_1,\ldots,\g_k$ and $\d_1,\ldots,\d_{k-1}$ are given for some
$k\ge 1$.
We choose a curve $\wt \gamma_{k+1}\in\Gamma^{(n_{k+1})}$ for some $n_{k+1}>n_k$
that intersects $\gamma_k$.
The existence of such a curve is plausible, but requires proof, and this is why
this construction is only a heuristic.
We observe that $\Delta_n(\Phi_{x,y})=0$ for all $n\ge n_k$ and $(x,y)\in\gamma_{k}$.
By continuity, there is a choice of $\d_{k}$ so that
$\Delta_{n}(\Phi_{x,y})\le \eta_n$ holds for all $n\in[n_k,n_{k+1})$ and
$(x,y)$ in the $\d_{k}$ neighbourhood of $\gamma_{k}$.
Finally we set $\gamma_{k+1}$ to be a suitable segment of $\wt\gamma_{k+1}$
contained in the $\d_{j}$ neighbourhood of $\gamma_j$ for all $j\le k$.

It is immediate from the construction that there is a point $(x,y)$ which
is contained in the (closed) $\delta_k$ neighbourhood of $\gamma_k$ for all $k$, and that
$\Delta_n(\Phi_{x,y})\le \eta_n$ for all $n$.

With a small modification of the construction, we can ensure that $\Phi_{x,y}$ contains
no exact overlaps for the resulting parameter point $(x,y)$.
Indeed, observe that $\bigcup\Gamma^{(n)}$ is a countable set,
and let $\g_1^*,\g_2^*,\ldots$
be an enumeration of it.
In the construction, we have considerable liberty in choosing the curve segment $\gamma_k$ so we can
make sure that it does not intersect $\g_k^*$.
(This requires, in particular, that we choose $\wt \g_k$ not to coincide with $\g_k^*$.
The possibility of this is again plausible, but requires proof.)
Then in the next step of the construction we can ensure that $\delta_k$ is chosen to be
sufficiently small so that $\g_k^*$ is entirely outside the $\delta_k$ neighbourhood
of $\gamma_k$.
This way we can ensure that the resulting parameter point $(x,y)$ at the end of the
process is not contained in $\g_k^*$ for any $k$, and hence $\Phi_{x,y}$ is
without exact overlaps.

%%%%%%%%%%%%%%%%%%%%%%%%%%%%%%%%%%%%%%%%%%%%%%%%%%%%%%%%%%%%%%%%%%%%%%%%%%%%%%%%
\section{IFS's with algebraic contraction factors}\label{sc:Rapaport}
%%%%%%%%%%%%%%%%%%%%%%%%%%%%%%%%%%%%%%%%%%%%%%%%%%%%%%%%%%%%%%%%%%%%%%%%%%%%%%%%

In this section we discuss the following result of Rapaport \cite{Rap-algebraic}.

\begin{thm}[Rapaport]\label{th:Rapaport}
Conjectures \ref{cn:EO-set} and \ref{cn:EO-measure} hold for all IFS's
in which all contraction parameters are algebraic numbers.
\end{thm}

This is a far reaching common generalization of Hochman's Corollaries \ref{cr:algebraic} and \ref{cr:Sierpinski}.
We discuss some of the main ideas in the special case of the family of IFS's
\[
\Phi_{s,t}=\Big\{x\mapsto \frac{x}{3},x\mapsto \frac{x}{3}+1, x\mapsto \frac{x}{3}+s,
x\mapsto \frac{x}{3}+t\Big\}
\]
with uniform probability weights.
This is perhaps the simplest family not contained in the results of Hochman, and as
was shown by Chen (see Section \ref{sc:failure}),
this family contains IFS's without exact overlaps that
fail the exponential separation property (in a very strong sense).

Let $\xi_1,\xi_2,\ldots$ be a sequence of independent random variables taking the
values $0,1,s,t$ with equal probabilities.
As we discussed in Section \ref{sc:Bernoulli}, the exponential separation
property can be relaxed in Hochman's results.
Instead of a lower bound on $\Delta_n$, it suffices to have bounds of the form
\begin{equation}\label{eq:ent-bound2}
H\Big(\sum_{j=0}^{n-1}\xi_j\cdot 3^{-j}; C^{-n}\Big)\ge(\log 3-\e_n) n
\end{equation}
for infinitely many values of $n$ with some constant $C$ and a sequence $\e_n\to 0$.
(See Section \ref{sc:Bernoulli} for the definition of this notation.)

Theorem \ref{th:Rapaport} is proved by verifying condition \eqref{eq:ent-bound2}.
With this aim in mind, we examine what happens when \eqref{eq:ent-bound2}
fails for some $n$, $C$ and $\e_n$.
We write $\cL^{(n)}$ for the family of (inhomogeneous)
linear forms of the form $a_1\cdot 1+a_2 Y_1+a_3 Y_2$,
where each $a_i$ is a sum of a subset of the numbers $1,3^{-1},\ldots, 3^{-n+1}$
and each term $3^j$ is allowed in at most one of the $a_i$.
This definition is designed so that the values taken by the random variable
$\sum_{j=0}^{n-1}\xi_j \cdot 3^{-j}$ are precisely the values of the linear forms in $\cL^{(n)}$
evaluated at $s$ and $t$.

We write $\cL^{(n)}-\cL^{(n)}$ for the set of linear forms that can be written as the difference of
two elements of $\cL^{(n)}$.
We also fix some parameter point $(s_0,t_0)$ such that the IFS lacks exact overlaps.
We consider pairs of elements in the support of $\sum_{j=0}^{n-1}\xi_j \cdot 3^{-j}$
that are at distance no more than $C^{-n}$.
Then for any  such pair,
there corresponds a linear form $L\in \cL^{(n)}-\cL^{(n)}$ such that $|L(s_0,t_0)|\le C^{-n}$.
We write $\cA^{(n)}$ for the collection of linear forms in $\cL^{(n)}-\cL^{(n)}$ that arise in this way.
(This definition depends on $C$, $s_0$ and $t_0$, which we suppress in our notation.)

Let $n$ be such that \eqref{eq:ent-bound2} fails (for some choice of $\e_n$ and $C$).
We distinguish two cases depending on the rank of $\cA^{(n)}$.
The first case arises when there are at least two linearly independent forms in $\cA^{(n)}$,
and the second case is when the elements of $\cA^{(n)}$ are all scalar multiples of each other.

In the first case, we take two linearly independent $L_1,L_2\in\cL^{(n)}-\cL^{(n)}$.
Provided $C$ is sufficiently large, the lines determined by $L_1$ and $L_2$ cannot be parallel.
Indeed, if that was the case, their distance would be a rational number with denominator
bounded by an exponential in $n$, which we can force to be $0$ by taking $C$ sufficiently large.
Since the lines are not parallel, we can solve the equations
\begin{align*}
L_1(s_n,t_n)=&0,\\
L_2(s_n,t_n)=&0,
\end{align*}
and find that its solution $(s_n,t_n)$ is a pair of rational numbers with denominators bounded
by an exponential in $n$.
Moreover, the distance of $(s_n,t_n)$ from $(s_0,t_0)$ will be an arbitrarily small exponential in $n$
if $C$ is chosen sufficiently large.

The points $(s_n,t_n)$ have the same repellency property as those in the proof of
Corollary \ref{cr:Sierpinski}.
We discuss next how to show
that the second case, that is when the elements of $\cA^{(n)}$ are proportional,
arises for only finitely many values of $n$.
Then the argument for Corollary \ref{cr:Sierpinski} can be carried over
to prove \eqref{eq:ent-bound2}.

We begin by extending the definition of entropy rates.
Let $\ell$ be a line in $\R^2$ (that does not necessarily contain $0$).
We denote by $Y_{\ell}^{(n)}$ the random $\ell\to \R$
function $(s,t)\mapsto\sum_{j=0}^{n-1}\xi_j(s,t)\cdot 3^{-j}$.
We define the entropy rate of the line $\ell$ by
\[
h(\ell):=\lim_{n\to\infty} \frac{H(Y_{\ell}^{(n)})}{n}.
\]
Here $H(Y_{\ell}^{(n)})$ stands for the Shannon entropy of $Y_{\ell}^{(n)}$, which is a random element taking finitely
many values.
It can be shown that $H(Y_{\ell}^{(n)})$ is subadditive, hence the limit exists and is equal to the
infimum.
The quantity $h(\ell)$ measures the amount of exact overlaps that occur simultaneously for
all parameter points $(s,t)\in\ell$.

Now suppose that the second case occurs for some $n$ in our above discussion, that is
the linear forms in $\cA^{(n)}$ are proportional.
Let $\ell$ be the line on which all elements of $\cA^{(n)}$ vanish.
It is immediate from the definition of $\cA^{(n)}$ that
\[
H(Y_{\ell}^{(n)})\le H\Big(\sum_{j=0}^n\xi_j 3^{-j};C^{-n}\Big).
\]
Supposing
\begin{equation}\label{eq:ent-bound3}
H\Big(\sum_{j=0}^n\xi_j 3^{-j};C^{-n}\Big)\le (\log 3-\e)n
\end{equation}
for some $\e>0$, we can conclude
\[
h(\ell)\le \log 3-\e.
\]

In light of all this, the next proposition -- implicit in \cite{Rap-algebraic} --
implies that the second case and \eqref{eq:ent-bound3} for
some fixed $\e>0$ may occur for only finitely many $n$'s.

\begin{prp}\label{pr:Rapaport}
Let $(s_0,t_0)$ be some parameters such that the IFS $\Phi_{s_0,t_0}$ contains no exact overlaps.
Fix some $\e>0$.
Then there is a neighbourhood of $(s_0,t_0)$ that is not intersected by any lines $\ell$ with
$h(\ell)\le\log 3-\e$.
\end{prp}

We end this section by discussing the proof of this result.
Suppose to the contrary that the result is false, that is, there is a sequence
$\ell_1,\ell_2,\ldots$
of lines passing closer and closer to $(s_0,t_0)$ with $h(\ell_n)<\log 3-\e$.
We suppose as we may that the lines $\ell_n$ converge (in any reasonable topology)
to a line $\ell_\infty$.
We also suppose for simplicity that none of $\ell_1,\ell_2,\ldots,\ell_\infty$
is parallel to either of the $s$ or $t$ axes, and none of them goes through the origin.

We associate a self-similar measure in $\R^2$ to each line $\ell_j$.
For $j=1,2,\ldots,\infty$, let $\s_j$ and $\tau_j$ be the unique numbers such that
$\ell_j$ is spanned by $(\s_j,0)$ and $(0,\tau_j)$.
For $\s,\tau\in\R$, we define the IFS
\begin{align*}
\Psi_{\s,\tau}:=
\Big\{&
(x,y)\mapsto \Big(\frac{x}{3},\frac{y}{3}\Big),
(x,y)\mapsto \Big(\frac{x}{3}+1,\frac{y}{3}+1\Big),\\
&(x,y)\mapsto \Big(\frac{x}{3}+\s,\frac{y}{3}\Big),
(x,y)\mapsto \Big(\frac{x}{3},\frac{y}{3}+\tau\Big)
\Big\},
\end{align*}
and write $\nu_{\s,\tau}$ for the associated self-similar measure
(with equal probability weights).

It is immediate from the definitions that the same exact overlaps occur for the
random variables $Y_{\ell_j}^{(n)}$ as for the IFS $\Psi(\s_j,\tau_j)$.
It follows that
\[
h(\Psi_{\s_j,\tau_j})=h(\ell_j)\le \log 3-\e
\]
for $j<\infty$.
Using this, it can be shown that
\[
\dim\nu_{\s_j,\tau_j}\le \frac{\log 3 -\e}{\log 3}= 1-\e/\log 3.
\]

It is a general phenomenon that the dimension of self-similar measures depends lower semi-continuously on the
parameters, see e.g. \cite{Fen} for results of this type covering even self-affine measures.
Using this, it follows that
\[
\dim\nu_{\s_\infty,\tau_\infty}\le 1-\e/\log 3.
\]

The proof of Proposition \ref{pr:Rapaport} is now finished by establishing a suitable
analogue of Conjecture \ref{cn:EO-measure} for the IFS's $\Psi_{\s,\tau}$, which shows
that $\Psi_{\s_0,\t_0}$ and hence $\Phi_{s,t}$ for all $(s,t)\in\ell$ including $(s_0,t_0)$
contains exact overlaps.
This can be done along the lines of the proof of Corollary \ref{cr:algebraic} discussed
in Section \ref{sc:exp-sep} using a higher dimensional version of Hochman's theorem,
which can be found in \cite{Hoc-Rd}.
The crucial difference between the IFS's $\Phi_{s,t}$ and $\Psi_{\s,\tau}$ is
that the ambient space is $2$-dimensional for the latter and this matches
the number of parameters.
This means that exact overlaps occur at single points (as opposed to along lines),
which have the required repellency property.

%%%%%%%%%%%%%%%%%%%%%%%%%%%%%%%%%%%%%%%%%%%%%%%%%%%%%%%%%%%%%%%%%%%%%%%%%%%%%%%
\section{Homogeneous IFS's of three maps} \label{sc:3maps}
%%%%%%%%%%%%%%%%%%%%%%%%%%%%%%%%%%%%%%%%%%%%%%%%%%%%%%%%%%%%%%%%%%%%%%%%%%%%%%%

In this section, we discuss the IFS's
\[
\Phi_{\l,t}=\{(x\mapsto\l x,x\mapsto \l x+1,x\mapsto\l x+t)\}.
\]
Rapaport and Varj\'u \cite{RV-3maps} made some partial progress
towards extending the results for Bernoulli convolutions
discussed in Section \ref{sc:Bernoulli} to this setting
and to some more general IFS's (see \cite{RV-3maps}*{Section 3}).

Before we can state these results, we need to introduce some relevant
notation and terminology.
We write $\mu_{\l,t}$ for the self-similar measure associated to the IFS $\Phi_{\l,t}$
and uniform probability weights.
Let $\xi_1,\xi_2,\ldots$ be a sequence of independent random $\R\to \R$ functions
taking the values $t\mapsto0$, $t\mapsto1$ and $t\mapsto t$ with
equal probability.
Let $U\subset(0,1)\times\R$, $n\in\Z_{\ge 0}$, and write $A_{U}^{(n)}$ for the 
random $U\to \R$ function
\[
(\l,t)\mapsto \sum_{j=1}^n \xi_j(t)\l^j.
\]
We define the entropy rate
\[
h(U):=\lim_{n\to \infty}\frac{H(A_{U}^{(n)})}{n}= \inf\frac{H(A_{U}^{(n)})}{n}.
\]
We abbreviate $A_{\{\l,t\}}^{(n)}$ as $A_{\l,t}^{(n)}$,
and $h(\{\l,t\})$ as $h(\l,t)$.
One should think about $h(\l,t)$ as a quantity expressing the amount of exact overlaps
contained in the the IFS $\Phi_{\l,t}$ and $h(U)$ aims to quantify the amount of
exact overlaps occurring simultaneously for the parameter points in $U$.

We write $\cR$ for the set of meromorphic functions on the unit disc
that can be written as ratios of two power series with coefficients $-1,0,1$. 
We denote by $\Gamma$ the set of curves $\gamma\subset(0,1)\times \R$ that are
either of the following two forms
\begin{itemize}
\item $\gamma=\{(\l,t)\in(0,1)\times \R: t=R(\l)\}$  for some $R\in\cR$,
\item $\gamma=\{(\l_0,t):t\in\R\}$ for some fixed $\l_0\in(0,1)$.
\end{itemize}
It can be shown that exact overlaps occur in the family of IFS's $\Phi_{\l,t}$
along finite unions of curves in $\Gamma$, but not all elements of $\Gamma$ arises in this way.

The next result is an analogue of Theorem \ref{th:BV2} in the setting of the IFS
$\Phi_{\l,t}$.

\begin{thm}[Rapaport, Varj\'u]\label{th:RV1}
Suppose that Conjecture \ref{cn:EO-measure} does not hold for the IFS $\Phi_{\l,t}$
for some choice of parameters $\l$ and $t$.
Then for every
$\e>0$ and $N\ge 1$, there exist $n\ge N$ and $(\eta,s)\in(0,1)\times \R$
such that
\begin{enumerate}
\item $|\l-\eta|,|t-s|\le\exp(-n^{\e^{-1}})$,
\item $\frac{1}{n\log \eta^{-1}}H(A_{\eta,s}^{(n)})\le\dim\mu_{\l,t}+\e$,
\item $h(\gamma)\ge\min\{\log 3,\log \l^{-1}\}-\e$ for all
$\gamma\in\Gamma$ with $(\eta, s)\in\gamma$.
\end{enumerate}
\end{thm} 

Item (2) in the conclusion means that the IFS $\Phi_{\eta,s}$ contains enough overlaps after
$n$ iteration to force the dimension of $\mu_{\eta,s}$ below $\dim\mu_{\l,t}+\e$.
Item (3) in the conclusion implies that not all of these exact overlaps occur along
the same curve $\gamma$.
From these properties it can be deduced in particular that $\eta$ and $s$ are algebraic
numbers and roots of polynomials of low degree with small integer coefficients.
(For a precise statement, see \cite{RV-3maps}*{Theorem 1.3}.)
This yields a bound on the number of possible points that can arise as $(\eta,s)$
in the conclusion and together with Item (1), this shows that the
Hausdorff dimension of the set of exceptional parameters for which
Conjecture \ref{cn:EO-measure} fails is $0$.
This improves Hochman's bound, which is $1$, albeit that bound is given for the
stronger notion of packing dimension, which may exceed the Hausdorff dimension.

It is still an open problem whether an analogue of Theorem \ref{th:BV1} holds for the IFS
$\Phi_{\l,t}$.
One possible formulation is the following.

\begin{que}\label{q:RV}
Is it true that for all $\e>0$, there is $M$ such that the following holds?
Let $(\l,t)\in(\e,1-\e)\times \R$ be such that $h(\l,t)\le \min(\log 3,\log \l^{-1})-\e$
and $h(\gamma)\ge\min(\log 3,\log \l^{-1})-M^{-1}$ for all $\g\in\Gamma$ with $(\l,t)\in\gamma$.
Then $M(\l)\le M$.
\end{que}

We note that a condition about the entropy rate of curves passing through $(\l,t)$ is
necessary.
Indeed, we have, for example $h(\gamma)=\log 3-(2/3)\log 2$ for the
curve $\gamma=\{(\l,1):\l\in(0,1)\}$, and hence $h(\l,1)\le\log 3-(2/3)\log 2 $
for all $\l\in(0,1)$.

We also have the following conditional result towards Conjecture \ref{cn:EO-measure}.

\begin{thm}[Rapaport, Varj\'u]\label{th:RV2}
Suppose that the answer to Question \ref{q:RV} is affirmative.
Then Conjecture \ref{cn:EO-measure} holds for the IFS $\Phi_{\l,t}$
with equal probability weights
for all $\l\in(0,1)$ and $t\in\R$.
\end{thm}

Using ideas from \cite{BV-entropy}, one can answer
Question \ref{q:RV} affirmatively if we restrict $\l$ to be near $1$.
This allows for the following unconditional partial resolution of Conjecture \ref{cn:EO-measure}.

\begin{thm}[Rapaport, Varj\'u]\label{th:RV3}
Conjecture \ref{cn:EO-measure} holds for the IFS $\Phi_{\l,t}$ with equal probability weights
for all $(\l,t)\in(2^{-2/3},1)\times\R$.
\end{thm} 

The key new ingredient in the proof of Theorem \ref{th:RV1} compared
to that of Theorem \ref{th:BV2}
is the following result, whose role is similar to that of
Proposition \ref{pr:Rapaport} in the proof of Theorem \ref{th:Rapaport}.

\begin{prp}\label{pr:RV}
Let $(\l,t)\in(0,1)\times \R$ be such that the IFS $\Phi_{\l,t}$ contains no exact overlaps.
Then for all $h<\min(\log\l^{-1},\log 3)$, there is a neighbourhood of $(\l,t)$ that
is not intersected by a curve $\gamma\in\Gamma$ with $h(\g)\le h$.
\end{prp}

The proof of this result like Proposition \ref{pr:Rapaport} is done by attaching suitable fractal
objects to curves and relating their dimension to the entropy rates of the curves.
Then the proposition is proved using lower semi-continuity of dimension and a
limiting argument.
The fractal measures used in the paper \cite{RV-3maps} are analogues of self-similar measures
in function fields.
A suitable notion of dimension is introduced for these objects and Hochman's theorem
is generalized to this setting.
The analogue of the exponential separation property is verified using an argument
similar to that used in the proof of Corollary \ref{cr:Sierpinski}.
An additional difficulty compared to the setting of Section \ref{sc:Rapaport} is
caused by the fact that the curves in $\Gamma$ are not necessarily lines and they may develop
singularities, which complicates limiting arguments.

The proofs of Theorems \ref{th:RV2} and \ref{th:RV3} is
complicated by the fact that like in the
case of Bernoulli convolutions, the parameter points with exact overlaps have a weaker
than exponential repellency property.
To address this, an argument similar to that discussed at the end of
Section \ref{sc:Bernoulli} is used.
This is the reason why we need to assume an affirmative answer to
Question \ref{q:RV}.
The argument also requires a stronger form of Proposition \ref{pr:RV}
with a modified entropy rate.
The precise statement requires some preparation.
For this reason, we omit it and refer to \cite{RV-3maps}*{Proposition 2.4}.

%%%%%%%%%%%%%%%%%%%%%%%%%%%%%%%%%%%%%%%%%%%%%%%%%%%%%%%%%%%%%%%
\section{Other developments}\label{sc:other}
%%%%%%%%%%%%%%%%%%%%%%%%%%%%%%%%%%%%%%%%%%%%%%%%%%%%%%%%%%%%%%%

We survey some recent results about aspects of self-similar measures
other than their dimensions.
Due to limitation of space, our discussion will be very brief.

%%%%%%%%%%%%%%%%%%%%%%%%%%%%%%%%%%%%%%%%%%%%%%%%%%%%%%%%%%%%%%%
\subsection{Fourier decay}
%%%%%%%%%%%%%%%%%%%%%%%%%%%%%%%%%%%%%%%%%%%%%%%%%%%%%%%%%%%%%%%

We first discuss Fourier decay of self-similar measures.
Specifically, we discuss the following three properties.
\begin{itemize}
	\item A measure $\mu$ on $\R$ is Rajchman if its Fourier transform
	vanishes at infinity, that is
	\[
	\lim_{|\xi|\to \infty}|\wh\mu(\xi)|=0.
	\]
	\item A measure $\mu$ on $\R$ has polylogarithmic Fourier decay if
there is a constant $a>0$ such that for all sufficiently large $\xi$, we have
\[
|\wh\mu(\xi)|<|\log |\xi||^{-a}.
\]
\item A measure $\mu$ on $\R$ has power Fourier decay if
there is a constant $a>0$ such that for all sufficiently large $\xi$, we have
\[
|\wh\mu(\xi)|<|\xi|^{-a}.
\]
\end{itemize}

There are various motivations for studying these properties.
The Rajchman property is closely related to an old subject in the theory
of trigonometric series about so-called sets of uniqueness and sets of multiplicity,
see \cite{KL} for more.
Fourier decay has also applications in metric Diophantine approximation.
For example, polylogarithmic Fourier decay is sufficient to guarantee
that almost all numbers with respect to the measure are normal in every bases.
(In the case of self-similar measures on $\R$, even the Rajchman property is
enough for this, see \cite{ARHW}*{Theorem 1.4}.)
Power decay is very useful in proving absolute continuity of the measure,
which we discuss more in the next section.

Results about these properties of self-similar measures come in two flavours.
In the first category, properties are proved for most self-similar measures
in a parametric family, in the second the properties are proved for explicit
self-similar measures, that is, the hypotheses of the results are testable in
concrete examples.

We begin by discussing results in the first category.
Erd\H os \cite{Erd-ac} proved that Bernoulli convolutions (see Section \ref{sc:Bernoulli})
have power Fourier decay for almost all choices of the parameter $\l\in(0,1)$.
His argument was revisited by Kahane \cite{Kah} who showed that the exceptional
set of parameters where the power decay fails is in fact of $0$ Hausdorff dimension.
This method was exposed in the survey \cite{PSS-60y}, where the exponent $a$
was also studied, and the term Erd\H os--Kahane argument was coined.
Recently Solomyak \cite{Sol-decay}
showed that non-degenerate self-similar measures on $\R$ has power Fourier decay
if the vector of contraction parameters avoid an exceptional set
of $0$ Hausodorff dimension.
See the references in \cite{Sol-decay}*{Section 1.1} and \cite{Sol-decay-sa}
for more recent applications of the Erd\H os--Kahane method.

The first results in the second category are also in the setting of Bernoulli convolutions.
Erd\H os \cite{Erd-Pisot}
proved that Bernoulli convolutions are not Rajchman when $\l^{-1}$, the
reciprocal of the parameter is a Pisot number, except when the probability weights
are uniform and $\l=1/2m$ for an odd integer $m$.
Recall that a Pisot number is an algebraic integer all of whose Galois conjugates lie
inside the complex unit disk.
Salem \cite{Sal} proved the converse of Erd\H os's
result by showing that Bernoulli convolutions
are Rajchman when $\l^{-1}$ is not Pisot.

The Rajchman property of general self-similar measures has been understood more
recently.
Sahlsten and Li \cite{LS} proved that self-similar measures are Rajchman whenever
the semigroup generated by the contraction parameters is not lacunary, that is, it
is not contained in $\{\l^n:n\in\Z_{\ge 0}\}$ for some $n$.
Their work is based on a new method relying on renewal theory originating in \cite{Li}.
See also \cite{ARHW}, where this result is extended to self-conformal measures
using a different method.
The lacunary case was analysed by Br\'emont \cite{Bre}, see also Varj\'u, Yu \cite{VY}. 
Finally, the problem was solved by Rapaport \cite{Rap-decay}
for self-similar measures on $\R^d$.

For Bernoulli convolutions, polylogarithmic Fourier decay  follows from
a result of Bufetov and Solomyak \cite[Proposition 5.5]{BS}
for algebraic parameters $\l$ provided $\l^{-1}$ is neither Salem or Pisot, that is, it has
another Galois conjugate outside the complex unit disk, see also \cite{GM}.
Under a mild Diophantine condition for the contraction parameters,
Sahlsten and Li \cite{LS} proved polylogarithmic Fourier decay for self-similar measures.
Informally speaking, their condition requires that the semigroup generated by the
contraction parameters is not approximated by lacunary semigroups in a suitable quantitative
sense.
See \cite{ARHW} for a similar result under a different Diophantine condition.
Polylogarithmic Fourier decay was also established by Varj\'u and Yu \cite{VY} for
certain self-similar measures in the lacunary case.

It is an important open problem to characterize which self-similar measures have
power Fourier decay.
Very little is known about this.
See \cite{DFW} for explicit examples of Bernoulli convolutions with power
Fourier decay and \cite{LV} for results about self-similar measures on $\R^d$ for $d\ge 3$.

%%%%%%%%%%%%%%%%%%%%%%%%%%%%%%%%%%%%%%%%%%%%%%%%%%%%%%%%%%%%%%%
\subsection{Absolute continuity}
%%%%%%%%%%%%%%%%%%%%%%%%%%%%%%%%%%%%%%%%%%%%%%%%%%%%%%%%%%%%%%%

Let $\mu$ be a self-similar measure on $\R$ associated to an IFS with contraction
factors $\{\l_i\}$ that conatins no exact overlaps, and probability weights $\{p_i\}$.
One may expect that $\mu$ is not only of dimension $1$ if
\begin{equation}\label{eq:simdim>1}
\frac{\sum p_i\log p_i^{-1}}{\sum p_i\log \l_i^{-1}}>1,
\end{equation}
as predicted by Conjecture \ref{cn:EO-measure}, but it is also absolutely continuous.
When there is equality in \eqref{eq:simdim>1}, the self-similar measure is almost
always singular, see \cite{NW}*{Theorem 1.1}.

In general, this expectation is false.
Simon and V\'ag\'o \cite{SV} showed that in some families of IFS's, there is a dense
$G^d$ set of parameters, which violate the above statement. 
See \cite{NPS} for earlier related results in a different setting.
However, it could still be true that \eqref{eq:simdim>1} and the lack of
exact overlaps imply absolute continuity
for some families of self-similar measures, for example for Bernoulli convolutions.

Nevertheless, it is expected that self-similar measures are absolutely continuous
for almost all choices of the parameters in parametric families
when \eqref{eq:simdim>1} holds.
For Bernoulli convolutions, this was proved by Erd\H os for $\l$ near $1$, as a
consequence of power Fourier decay with parameter $a>1$.
The result has been extended to the optimal range $\l\in[1/2,1]$ by Solomyak \cite{Sol}
using the transversality method.
See \cite{PSS-60y}, \cite{NW}, \cite{BSSS} and their references for further developments.
Shmerkin \cite{Shm-ac} proved that the set of exceptional parameters in $[1/2,1]$ that make
the Bernoulli convolution singular is of Hausdorff dimension $0$.
His method is based on a result of his that the convolution of a measure of dimension $1$
and another one with power Fourier decay is absolutely continuous.
He used this in conjunction with Hochman's theorem and the Erd\H os-Kahane method.
See \cite{Shm}, \cite{SSS} and the references therein
for further developments using this method.

Explicit examples of absolutely continuous self-similar measures are rare.
The first examples were given by Garsia \cite{Gar-arithmetic} as the Bernoulli
convolutions with parameters of Mahler measure $2$.
See \cite{DFW} for a generalization of this construction, and see \cite{Yu}
for an improvement on the regularity of the density function using Shmerkin's method.
Varj\'u gave new examples of absolutely continuous Bernoulli convolutions in
\cite{Var-ac}.
This paper relies on a similar method to Hochman's in a quantitatively refined form.
A crucial point is that it requires the separation condition to hold at all
sufficiently small scales rather than just at infinitely many of them.
This restricts the method to algebraic parameters currently.
A recent improvement was given by Kittle \cite{Kit}, who gave further new examples
of absolutely continuous Bernoulli convolutions.
While all the new examples in \cite{Var-ac} are very close to $1$, e.g.~$1-10^{-50}$,
this is not the case for \cite{Kit}, which includes e.g.~one near $0.799533\ldots$.
The paper \cite{Kit} also introduces a new tool to
quantify the smoothness of measures at scales.

See \cite{LV} for results about absolute continuity of self-similar measures on $\R^d$
for $d\ge 3$.

%%%%%%%%%%%%%%%%%%%%%%%%%%%%%%%%%%%%%%%%%%%%%%%%%%%%%%%%%%%%%%%%%%%%%%%%%%%%%%%%%%%%%%%%%
\section*{Acknowledgement}
I am grateful to
Simon Baker, Bal\'azs B\'ar\'any, Emmanuel Breuillard, Michael Hochman, Antti K\"aenm\"aki,
Samuel Kittle,  Ariel Rapaport, Tuomas Sahlsten,
Pablo Shmerkin, Boris Solomyak, Laurtiz Streck and Han Yu
for reading and commenting on an earlier version of this paper.
%%%%%%%%%%%%%%%%%%%%%%%%%%%%%%%%%%%%%%%%%%%%%%%%%%%%%%%%%%%%%%%%%%%%%%%%%%%%%%%%%%%%%%%%%

\bibliography{bibfile}

\bigskip

\noindent{\sc Centre for Mathematical Sciences,
Wilberforce Road, Cambridge CB3 0WA,
UK}\\
{\em e-mail address:} pv270@dpmms.cam.ac.uk

\end{document}